\documentclass[twoside,reqno,11pt]{fcaa}

\usepackage{graphicx}
\usepackage{epsfig}
\usepackage{amsthm}
\usepackage{amsmath}
\usepackage{latexsym}
\usepackage{amsfonts}
\usepackage{amssymb}

\newtheoremstyle{theorem}
  {15pt}          
  {15pt}  
  {\sl}  
  {\parindent}
  {\sc}  
  {. }   
  { }    
  {}     
\theoremstyle{theorem}
\newtheorem{lemma}{Lemma}[section]
\newtheorem{theorem}{Theorem}[section]

\newtheoremstyle{defi}
  {15pt}          
  {15pt}  
  {\rm}  
  {\parindent}     
  {\sc}  
  {. }    
  { }    
  {}     
\theoremstyle{defi}

\newtheorem{remark}{Remark}[section]

 \def\proofname{\indent\rm P\ r\ o\ o\ f.}
 
 \def\qed{\hfill$\Box$} 

\usepackage{hyperref}
\usepackage[capitalize]{cleveref}

\def\theequation{\arabic{section}.\arabic{equation}}
\def\themycopyrightfootnote{\vspace*{3pt}
\copyright \, Year\,  Diogenes  Co., Sofia
  \par  \noindent pp. xxx--xxx, DOI: ......................
  \hfill  \vspace*{-36pt}
}
\newcommand{\R}{\mathbb{R}}
\newcommand{\FI}{\mathcal{I}}
\newcommand{\iprod}[1]{\langle#1\rangle}
\newcommand{\bigiprod}[1]{\bigl\langle#1\bigr\rangle}
\newcommand{\biggiprod}[1]{\biggl\langle#1\biggr\rangle}
\newcommand{\Mult}{\mathcal{M}}
\newcommand{\vecspan}{\operatorname{span}}

\newcommand{\Q}{\mathcal{Q}}

\newcommand{\CFD}[2]{{}^{\mathrm{C}}\partial^{#1}_{#2}}
\newcommand{\veczero}{\boldsymbol{0}}
\title[WELL-POSEDNESS \dots]{WELL-POSEDNESS 
OF TIME-FRACTIONAL ADVECTION-DIFFUSION-REACTION EQUATIONS \\[3pt]
IN ``FCAA'' JOURNAL} 

\author[\normalsize W. McLean, K. Mustapha, R. Ali, O. Knio]{William McLean$^1$,
Kassem Mustapha$^2$, Raed Ali$^2$, Omar Knio$^3$}

\begin{document}

\vbox to 2.5cm { \vfill }
\bigskip \medskip

\begin{abstract}
We establish the well-posedness of an initial-boundary value problem for a 
general class of linear time-fractional, advection-diffusion-reaction 
equations, allowing space- and time-dependent coefficients as well as initial 
data that may have low regularity. Our analysis relies on novel energy 
methods in combination with a fractional Gronwall inequality and properties of 
fractional integrals.

\medskip
{\it MSC 2010\/}: Primary
26A33; 
Secondary
35A01, 
35A02, 
35B45, 
35D30, 
35K57, 
35Q84, 
35R11. 

\smallskip
{\it Key Words and Phrases}:
Fractional PDE, weak solution, Volterra integral equation, 
fractional Gronwall inequality, Galerkin method.
\end{abstract}

\maketitle
\vspace*{-16pt}
\section{Introduction}
The main scope of this paper is to investigate the existence and uniquesness 
of the weak solution of a linear, time-fractional problem of the form
\begin{equation}\label{eq: FPDE}
\partial_tu-\nabla\cdot\bigl(\kappa\nabla\partial_t^{1-\alpha}u
	-\vec F\partial_t^{1-\alpha}u-\vec Gu\bigr)+a\partial_t^{1-\alpha}u+bu=g
\end{equation}
for $x\in\Omega$ and $0<t\le T$. The parameter~$\alpha$ in the fractional
derivative lies in the range~$0<\alpha<1$, and the spatial 
domain~$\Omega\subseteq\R^d$ ($d\ge1$) is bounded and Lipschitz.  The 
coefficients $\vec F$, $\vec G$, $a$~and $b$, as well as the source 
term~$g$, are assumed to be known functions of $x$~and $t$, whereas
the generalized diffusivity~$\kappa=\kappa(x)$ may depend only on~$x$
but is permitted to be a real, symmetric positive-definite matrix.  
We impose homogeneous Dirichlet boundary conditions,
\begin{equation}\label{eq: Dirichlet bc}
u(x,t)=0\quad\text{for $x\in\partial\Omega$ and $0\le t\le T$,}
\end{equation}
and the initial condition
\begin{equation}\label{eq: initial condition}
u(x,0)=u_0(x)\quad\text{for $x\in\Omega$.}
\end{equation}
The Riemann--Liouville fractional \emph{derivative} of
order~$1-\alpha$ is defined via the fractional \emph{integral} of 
order~$\alpha$: with $\omega_\alpha(t)=t^{\alpha-1}/\Gamma(\alpha)$ we have
\[
\partial_t^{1-\alpha}v(x,t)=\partial_t \FI^\alpha v(x,t)
\quad\text{where}\quad
	\FI^\alpha v(x,t)=\int_0^t\omega_\alpha(t-s)v(x,s)\,ds.
\]

We denote by~$W^k_p(\Omega)$ the usual Sobolev space of functions whose partial 
derivatives of order~$k$ or less belong to~$L_p(\Omega)$. The following 
regularity assumptions on the coefficients will be used:
\begin{equation}\label{eq: reg coeff} 
\begin{gathered}
\kappa\in L_\infty(\Omega)^{d\times d},\qquad
\vec F, \vec G\in C^2\bigl([0,T];W^1_\infty(\Omega)^d\bigr),\\
a, b\in C^1\bigl([0,T];L_\infty(\Omega)\bigr).
\end{gathered}
\end{equation}
In addition, to ensure that the spatial 
operator~$v\mapsto-\nabla\cdot(\kappa\nabla v)$ is uniformly elliptic 
on~$\Omega$, we assume that the minimal eigenvalue of~$\kappa(x)$ is bounded 
away from zero, uniformly for~$x\in\Omega$.

Different classes of time-fractional models (typically considered only for 
scalar~$\kappa$) arise as special cases of~\eqref{eq: FPDE}, including
\begin{itemize}
\item fractional Fokker--Planck equations~\cite{AngstmannEtAl2017,
HenryLanglandsStraka2010,KlafterSokolov2011,MetzlerBarkaiKlafter1999}, when 
$\vec G=\veczero$, 
$a=b=0$ and $g=0$; \item fractional reaction-diffusion 
equations~\cite{HenryLanglandWearne2006,HenryWearne2000},
when $\vec F=\vec G=\veczero$;
\item fractional cable equations~\cite{LanglandsHenryWearne2011}, 
when $\vec F=\vec G=\veczero$;
\item fractional advection-dispersion (or fractional convection-diffusion) 
equations~\cite{LiuAnhTurnerZhuang2003}, when  $\vec F=\vec F(x)$, 
$\vec G=\veczero$ and $a=b=0$.
\end{itemize}
Consider the simplest non-trivial case, when $\kappa$ is the identity 
matrix with $\vec F=\vec G=\veczero$, $a=b=0$~and $g=0$, so that 
\eqref{eq: FPDE} reduces to the fractional subdiffusion equation: 
$\partial_t u-\nabla^2\partial_t^{1-\alpha}u=0$. Let
$\varphi$ denote a Dirichlet eigenfunction of the Laplacian on~$\Omega$, with 
corresponding eigenvalue~$\lambda>0$, that is,
$-\nabla^2\varphi=\lambda\varphi$ in $\Omega$ with 
$\varphi|_{\partial\Omega}=0$.  For the special choice of initial data 
$u_0=\varphi(x)$, the solution of the initial-boundary value 
problem~\eqref{eq: FPDE}--\eqref{eq: initial condition} has the separable 
form~$u(x,t)=E_\alpha(-\lambda t^\alpha)\varphi(x)$, where
$E_\alpha(z)=\sum_{n=0}^\infty z^n/\Gamma(1+n\alpha)$ is the Mittag--Leffler
function.  Notice that $\partial_t^mu=O(t^{\alpha-m})$ as~$t\to0$.  Moreover, 
we can extend the classical method of separation of variables for the heat 
equation to construct a series solution for arbitrary initial 
data~$u_0\in L_2(\Omega)$,  and the regularity properties of the solution~$u$ 
follow from this representation~\cite{McLean2010}.

Such an explicit construction is no longer possible for the solution of 
the general equation~\eqref{eq: FPDE}. Instead, we proceed by formally 
integrating \eqref{eq: FPDE} in time, multiplying both sides by a test 
function~$v$, and applying the first Green identity over~$\Omega$ to
arrive at the weak formulation
\begin{multline}\label{eq: weak}
\iprod{u(t),v}
+\int_0^t\bigiprod{\kappa\nabla\partial_s^{1-\alpha}u(s)
	-\vec F(s)\partial_s^{1-\alpha}u(s)-\vec G(s)u(s),\nabla v}\,ds\\
+\int_0^t\bigiprod{a(s)\partial_s^{1-\alpha}u(s)+b(s)u(s),v}\,ds
	=\iprod{u_0,v}+\int_0^t\iprod{g(s),v}\,ds
\end{multline}
for all $v\in H^1_0(\Omega)$, where we have suppressed the dependence of the 
functions on~$x$, and where $\iprod{\cdot,\cdot}$ denotes the inner product in 
$L_2(\Omega)$~or $L_2(\Omega)^d$.

Numerical methods for particular cases of~\eqref{eq: FPDE} were extensively 
studied over the last two decades~\cite{CuestaLubichPalencia2006, JinLiZhou2018,
KaraaPani2018, LanglandsHenry2005, LeMcLeanMustapha2018, LiaoLiZhang2018, 
LinXu2007, Mustapha2015, StynesEtAl2017, YusteAcedo2005}. However, due to 
various types of mathematical difficulties, proof of the well-posedness
of the continuous problem is almost missing despite its importance, 
apart from the case~\cite{McLean2010} when $\vec F=\vec G=\veczero$ and 
$a=b=0$. In this paper, we address these fundamental questions.  A related
preprint~\cite{LeMcLeanStynes2018} treats the fractional Fokker--Planck equation
(that is, the case $\vec G=\veczero$ and $a=b=0$) via a different, and somewhat
simpler, chain of estimates that, for instance, does not use the quadratic 
operator~$\mathcal{Q}^\mu_1$ defined below in~\cref{sec: prelim}.

If the coefficients $\vec F$~and $a$ are independent of~$t$, and if 
$\vec G=\veczero$~and $b=0$, then by applying the fractional integration 
operator~$\FI^{1-\alpha}$ to both sides of~\eqref{eq: FPDE} we obtain
\begin{equation}\label{eq: alt FPDE}
\CFD{\alpha}{t}u-\nabla\cdot(\kappa\nabla u-\vec F u)+au=\tilde g,
\end{equation}
where $\CFD{\alpha}{t}u=\FI^{1-\alpha}\partial_tu$ denotes the Caputo 
fractional derivative and where $\tilde g=\FI^{1-\alpha}g$.  Existence 
and uniqueness results for~\eqref{eq: alt FPDE} were studied by several authors, 
including Zacher~\cite{Zacher2009}, Alikhanov~\cite{Alikhanov2010},
Sakamoto and Yamamoto~\cite{SakamotoYamamoto2011} and
Kubica and Yamamoto~\cite{KubicaYamamoto2018}.  Some of these papers include 
results for time-dependent coefficients, but in that case \eqref{eq: alt FPDE} 
is no longer equivalent to~\eqref{eq: FPDE}.

To recast the weak formulation~\eqref{eq: weak} as a Volterra 
integral equation, we introduce two bounded linear
operators, firstly $K_1(t):H^1_0(\Omega)\to H^{-1}(\Omega)$ defined by
\[
\iprod{K_1(t)v,w}=\iprod{\kappa\nabla v,\nabla w}
	-\iprod{\vec F(t)v,\nabla w}+\iprod{a(t)v,w}
\quad\text{for $v$, $w\in H^1_0(\Omega)$,}
\]
and secondly $K_2(t):L_2(\Omega)\to H^{-1}(\Omega)$ by
\[
\iprod{K_2(t)v,w}=\iprod{b(t)v,w}-\iprod{\vec G(t)v,\nabla w}
	\quad\text{for $v\in L_2(\Omega)$ and $w\in H^1_0(\Omega)$.}
\]
The variational problem~\eqref{eq: weak}, subject to the initial 
condition~\eqref{eq: initial condition}, can then be written more succinctly as
\begin{equation}\label{eq: Volterra eqn prelim}
u(t)+\int_0^t\bigl[K_1(s)\partial_s^{1-\alpha}u(s)+K_2(s)u(s)\bigr]\,ds=f(t)
\equiv u_0+\int_0^tg(s)\,ds.
\end{equation}
Integrating by parts, and using a dash to indicate a derivative in time, leads 
to
\begin{multline*}
\int_0^tK_1(s)\partial_s^{1-\alpha}u(s)\,ds=K_1(t)\FI^\alpha u(t)
	-\int_0^tK_1'(s)\FI^\alpha u(s)\,ds\\
=\int_0^t\biggl(\omega_\alpha(t-s)K_1(t)
	-\int_s^t\omega_\alpha(z-s)K_1'(z)\,dz\biggr)u(s)\,ds,
\end{multline*}
with~$K_1'(t):H^1_0(\Omega)\to H^{-1}(\Omega)$ given by
\[
\iprod{K_1'(t)v,w}=-\iprod{\vec F'(t)v,\nabla w}+\iprod{a'(t)v,w}.
\]
Thus, $u$ satisfies 
\begin{equation}\label{eq: Volterra eqn}
u(t)+\int_0^tK(t,s)u(s)\,ds=f(t)\quad\text{for $0\le t\le T$,}
\end{equation}
where~$K(t,s):H^1_0(\Omega)\to H^{-1}(\Omega)$ is the weakly-singular, 
operator-valued kernel
\begin{equation}\label{eq: K def}
K(t,s)=\omega_\alpha(t-s)K_1(t)+K_2(s)
	-\int_s^t\omega_\alpha(z-s)K_1'(z)\,dz.
\end{equation}

Following some technical preliminaries in \cref{sec: prelim}, we apply the 
Gal\-erkin method in \cref{sec: projected} to project the 
problem~\eqref{eq: Volterra eqn} to a finite dimensional
subspace~$X\subseteq H^1_0(\Omega)$, thereby obtaining an approximate
solution~$u_X:[0,T]\to X$. Using delicate energy arguments and a fractional
Gronwall inequality, we prove \emph{a priori} estimates for~$u_X$ that are
uniform with respect to the dimension of~$X$, allowing us in
\cref{sec: weak} (\cref{thm: existence,thm: uniqueness}) to establish the 
existence and uniqueness of a weak solution~$u$ to the original problem 
\eqref{eq: FPDE}--\eqref{eq: initial condition}, provided 
\eqref{eq: reg coeff} holds. 

The regularity of the weak solution~$u$ will be studied in a companion 
paper~\cite{McLeanEtAl2019}.
\section{Preliminaries and notations}\label{sec: prelim}
Our subsequent analysis makes frequent use of two quadratic operators
defined, for $\mu\ge0$ and $0\le t\le T$, by
\[
\Q^\mu_1(\phi,t)=\int_0^t\iprod{\phi,\FI^\mu\phi}\,ds
\quad\text{and}\quad
\Q^\mu_2(\phi,t)=\int_0^t\|\FI^\mu\phi\|^2\,ds.
\]
These operators coincide when~$\mu=0$ because $\FI^0\phi=\phi$, and so we write 
$\Q^0=\Q_1^0=\Q_2^0$. If we put $\phi(t)=0$ for~$t>T$, then the Laplace
transform~$\hat\phi(z)=\int_0^T e^{-zt}\phi(t)\,dt$ is an entire function 
and $\widehat{\FI^\mu\phi}(z)=z^{-\mu}\hat\phi(z)$, so it follows by the 
Plancherel Theorem that
\begin{equation}\label{eq: Q1 Plancherel}
\Q^\mu_1(\phi,T)=\frac{\cos(\pi\mu/2)}{\pi}
	\int_0^\infty y^{-\mu}\|\hat\phi(iy)\|^2\,dy\ge0,
\end{equation}
assuming that $\phi$ is real-valued; see also \cite[Theorem~2]{NohelShea1976}.
Note that because $\omega_\mu\in L_1(0,T)$, the fractional integral defines a 
bounded linear operator  
\begin{equation}\label{eq: I mu bounded}
\FI^\mu:L_p\bigl((0,T),L_2(\Omega)\bigr)\to 
	L_p\bigl((0,T),L_2(\Omega)\bigr)\quad\text{for~$1\le p\le\infty$.}
\end{equation}
Also, $\FI^{\mu+\nu}=\FI^\mu\FI^\nu$ because 
$\omega_\mu*\omega_\nu=\omega_{\mu+\nu}$ for $\mu>0$ and $\nu>0$; here, $*$
denotes the Laplace convolution.

The next four lemmas establish key inequalities satisfied by $\Q^\mu_1$~and
$\Q^\mu_2$.

\begin{lemma}\label{lem: alpha dep}
If $0<\alpha<1$ and $\epsilon>0$, then
\begin{gather}
\biggl|\int_0^t\iprod{\phi,\FI^\alpha\psi}\,ds\biggr|
	\le\frac{\Q^\alpha_1(\phi,t)}{4\epsilon(1-\alpha)^2}
	+\epsilon\,\Q_1^\alpha(\psi,t), \label{eq: A}\\
\Q^\alpha_2(\phi,t)\le\frac{2t^\alpha}{1-\alpha}\,
	\Q^\alpha_1(\phi,t), \label{eq: B}\\
\Q^\alpha_1(\phi,t)\le2t^\alpha\,\Q^0(\phi,t),\label{eq: C}\\
\biggl|\int_0^t\iprod{\phi,\FI^\alpha\psi}\,ds\biggr|
        \le\frac{t^\alpha\Q^0(\phi,t)}{2\epsilon(1-\alpha)^2}
        +\epsilon\,\Q^\alpha_1(\psi,t). \label{eq: AC}
\end{gather}
\end{lemma}
\begin{proof} 
The first three inequalities are proved by Le, McLean and 
Mustapha~\cite[Lemma 3.2]{LeMcLeanMustapha2018}. The fourth inequality follows 
from \eqref{eq: A} and \eqref{eq: C}.
\end{proof}

For the next result, note that if $\phi\in W^1_1\bigl((0,T);X\bigr)$ for a 
normed space~$X$, then $\phi:[0,T]\to X$ is absolutely continuous
and
\begin{equation}\label{eq: FI omega}
(\partial_t\FI^\alpha\phi-\FI^\alpha\partial_t\phi)(t)=\phi(0)\omega_\alpha(t)
\quad\text{for $0<t\le T$.}
\end{equation}

\begin{lemma}\label{lem: D}
If $0<\alpha\le1$, then for $\phi \in L_2\bigl((0,t),L_2(\Omega)\bigr)$,
\[
\Q^\alpha_2(\phi,t)\le2\int_0^t \omega_\alpha(t-s)\Q^\alpha_1(\phi,s)\,ds.
\]
\end{lemma}
\begin{proof}
Assume first that $\phi\in W^1_1\bigl((0,T),L_2(\Omega)\bigr)$ and let
$\psi=\FI^\alpha\phi$.  Since $\psi(0)=0$, the Caputo fractional 
derivative of~$\psi$ is
\[
\CFD{\alpha}{t}\psi=\FI^{1-\alpha}(\psi')
	=(\FI^{1-\alpha}\psi)'-\psi(0)\omega_{1-\alpha}=(\FI^1\phi)'=\phi.
\]
Recalling an identity of Alikhanov~\cite[Corollary1]{Alikhanov2012},
\begin{multline*}
2\bigiprod{\psi(t),\CFD{\alpha}{t}\psi(t)}
	=\CFD{\alpha}{t}\bigl(\|\psi\|^2\bigr)(t)\\
	+\frac{\alpha}{2\Gamma(1-\alpha)}\int_0^t\,\frac{1}{(t-s)^{1-\alpha}}
	\biggl(\int_0^s\frac{\psi'(q)\,dq}{(t-q)^\alpha}\biggr)^2\,ds,
\end{multline*}
we see that
\begin{equation}\label{eq: Alik1}
2\iprod{\phi,\FI^\alpha\phi}=2\iprod{\CFD{\alpha}{t}\psi,\psi}
	\ge\CFD{\alpha}{t}\bigl(\|\psi\|^2\bigr)
	=\FI^{1-\alpha}(\|\FI^\alpha\phi\|^2)',
\end{equation}
and thus 
\begin{align*}
\FI^1\bigl(\|\FI^\alpha\phi\|^2\bigr)
&=\FI^2\bigl(\|\FI^\alpha\phi\|^2\bigr)'
=\FI^{1+\alpha}\FI^{1-\alpha}\bigl(\|\FI^\alpha\phi\|^2\bigr)'\\
&\le2\FI^{1+\alpha}\bigl(\iprod{\phi,\FI^\alpha\phi}\bigr)
=2\FI^\alpha\FI^1\bigl(\iprod{\phi,\FI^\alpha\phi}\bigr),
\end{align*}
which is equivalent to the desired inequality.

Now let~$\phi\in L_2\bigl((0,T),L_2(\Omega)\bigr)$, and choose  
$\phi_n\in W^1_1\bigl((0,T),L_2(\Omega)\bigr)$ such that
$\int_0^T\|\phi_n(t)-\phi(t)\|^2\,dt\to0$ as~$n\to\infty$.  Using 
\eqref{eq: I mu bounded} with $\mu=\alpha$~and $p=2$, it follows that 
$\Q^\alpha_1(\phi_n,t)\to\Q^\alpha_1(\phi,t)$ and
$\Q^\alpha_2(\phi_n,t)\to\Q^\alpha_2(\phi,t)$, uniformly for~$t\in[0,T]$, 
which implies the result in the general case.
\end{proof}

The next lemma will eventually enable us to establish pointwise (in 
time) estimates for~$u(t)$.

\begin{lemma}\label{lem: pointwise bound}
Let $0<\alpha\le 1$. If the function 
$\phi\in W^1_1\bigl((0,t);L_2(\Omega)\bigr)$ satisfies
$\phi(0)=\FI^\alpha\phi'(0)=0$, then
$\|\phi(t)\|^2\le2\omega_{2-\alpha}(t)\,\Q_1^\alpha(\phi',t)$.
\end{lemma}
\begin{proof} For $\alpha=1$, equality holds:
\[
2\omega_1(t)\Q_1^1(\phi',t)=2\int_0^t\iprod{\phi',\phi}\,ds=\|\phi(t)\|^2.
\]
For $0<\alpha<1$, applying the operator $\FI^1$ to both sides 
of~\eqref{eq: Alik1} with $\phi'$ in place of $\phi$, and using 
$\FI^\alpha\phi'(0)=0$, we observe that,
\begin{align}\label{eq: new bound}
\FI^{1-\alpha}\big(\|\FI^\alpha\phi'\|^2\big)(t)&\le 2\Q_1^\alpha(\phi',t).
\end{align}
Put $\psi(t)=\FI^\alpha\phi'$. Since
$\phi=\FI^1\phi'=\FI^{1-\alpha}\psi$,
\begin{align*}
\|\phi(t\|^2&\le\biggl(\int_0^t\omega_{1-\alpha}(t-s)\|\psi(s)\|\,ds
	\biggr)^2\\
	&\le\int_0^t\omega_{1-\alpha}(t-s)
	\,ds\int_0^t\omega_{1-\alpha}(t-s)\|\psi(s)\|^2\,ds\\
	&=\omega_{2-\alpha}(t)\,\FI^{1-\alpha}\Big(\|\FI^\alpha\phi'\|^2\Big)(t),
\end{align*}
and hence the  desired result follows immediately after using
\eqref{eq: new bound}.  
\end{proof}

\begin{lemma}\label{lem: E}
If $0\le\mu\le\nu\le1$, then
$\Q^\nu_2(\phi,t)\le 2t^{2(\nu-\mu)}\Q^\mu_2(\phi,t)$.
\end{lemma}
\begin{proof}
See Le, McLean and Mustapha~\cite[Lemma~3.1]{LeMcLeanMustapha2018}.
\end{proof}

We will make essential use of the following fractional Gronwall inequality.

\begin{lemma}\label{lem: Gronwall}
Let $\beta>0$ and $T>0$.  Assume that $\mathsf{a}$ and $\mathsf{b}$ are
non-negative, non-decreasing functions on the interval~$[0,T]$.  If
$\mathsf{q}:[0,T]\to\R$ is an integrable function satisfying
\[
0\le\mathsf{q}(t)\le\mathsf{a}(t)
	+\mathsf{b}(t)\int_0^t\omega_\beta(t-s)\mathsf{q}(s)\,ds
    \quad\text{for $0\le t\le T$,}
\]
then
\[
\mathsf{q}(t)\le\mathsf{a}(t)E_\beta\bigl(\mathsf{b}(t)t^\beta\bigr)
    \quad\text{for $0\le t\le T$.}
\]
\end{lemma}
\begin{proof}
See Dixon and McKee~\cite[Theorem~3.1]{DixonMcKee1986}.
\end{proof}

Let $\Mult$ denote the operator of pointwise multiplication by~$t$, that 
is, $(\Mult\phi)(t)=t\phi(t)$, and note the commutator property
\begin{equation}\label{eq: commutator}
\Mult\FI^\mu-\FI^\mu\Mult=\mu\FI^{\mu+1},
\end{equation}
for any real~$\mu\ge0$.  We will need the following estimates involving the 
linear operator~$B^\mu_\psi$ defined (for suitable $\psi$~and $\phi$) by
\begin{equation}\label{eq: Yd B}
(B^\mu_\psi\phi)(t)=\psi(t)\,\FI^\mu\phi(t)
	-\int_0^t \psi'(s)\,\FI^\mu\phi(s)\,ds.
\end{equation}

\begin{lemma} If $\psi\in W^1_\infty\bigl((0,T);L_\infty(\Omega)^d\bigr)$
and $\phi\in W^1_1\bigl((0,T);L_2(\Omega)\bigr)$, then there is a constant~$C$ 
(depending only on $\psi$, $\mu$~and $T$) such that for~$0\le t\le T$,
\begin{gather}
\Q^0(B_\psi^\mu\phi,t)\le C\Q^\mu_2(\phi,t),
	\label{eq: B1 estimate int}\\
\Q^0(\Mult B_\psi^\mu\phi,t)+\Q^0(\FI^1 B_\psi^\mu\phi,t)
	\le Ct^2\Q^\mu_2(\phi,t), \label{eq: B3 estimate int}\\
\Q^0\bigl((\Mult B^\mu_\psi\phi)',t\bigr)
    \le C\Q^\mu_2\bigl((\Mult\phi)',t\bigr)+C\Q^\mu_2(\Mult\phi,t)
	+C\Q^\mu_2(\phi,t). \label{eq: B5 estimate int}
\end{gather}\end{lemma}
\begin{proof}
The assumption on~$\psi$ implies that
\[
\|(B^\mu_\psi\phi)(t)\|^2\le C\|(\FI^\mu\phi)(t)\|^2
	+C\int_0^t\|(\FI^\mu\phi)(s)\|^2\,ds,
\]
and \eqref{eq: B1 estimate int} follows after integrating in time.  By the
Cauchy--Schwarz inequality,
\[
\|(\Mult B_\psi^\mu\phi)(t)\|^2+\|(\FI^1 B_\psi^\mu\phi)(t)\|^2
	\le t^2\|(B_\psi^\mu\phi)(t)\|^2+t\int_0^t\|(B_\psi^\mu\phi)(s)\|^2\,ds,
\]
and \eqref{eq: B3 estimate int} follows after integrating in time.  The
third identity in~\eqref{eq: commutator} implies that
\[
\Mult B^\mu_\psi\phi=\psi\bigl(\FI^\mu\Mult\phi+\mu\FI^{\mu+1}\phi\bigr)
	-\Mult\FI^1(\psi'\FI^\mu\phi)
\]
and therefore, differentiating with respect to~$t$,
\[
(\Mult B_\psi^\mu\phi)'=\psi'\bigl(\FI^\mu\Mult\phi+\mu\FI^{\mu+1}\phi\bigr)
	+\psi\bigl((\FI^\mu\Mult\phi)'+\mu\FI^\mu\phi\bigr)
    -(\FI^1+\Mult)(\psi'\FI^\mu\phi).
\]
Thus, noting that $(\FI^\mu\Mult\phi)'=\FI^\mu(\Mult\phi)'$
by~\eqref{eq: FI omega}, with
\[
\|\FI^{\mu+1}\phi(t)\|^2=\|\FI^1(\FI^\mu\phi)(t)\|^2
\le t\Q^\mu_2(\phi,t)
\]
and $\|\FI^1(\psi'\FI^\mu\phi)(t)\|^2\le Ct\Q^\mu_2(\phi,t)$, we have
\begin{align*}
\|(\Mult B_\psi^\mu\phi)'(t)\|^2&\le C\|\FI^\mu(\Mult\phi)(t)\|^2
	+C\|\FI^\mu(\Mult\phi)'(t)\|^2\\
	&\qquad{}+C\|(\FI^\mu\phi)(t)\|^2+Ct\Q^\mu_2(\phi,t),
\end{align*}
so \eqref{eq: B5 estimate int} follows after integrating in time.
\end{proof}

\section{The projected equation}\label{sec: projected}

Suppose that $X$ is a finite-dimensional subspace of~$H^1_0(\Omega)$, equipped
with the induced norm: $\|v\|_X=\|v\|_{H^1_0(\Omega)}$.
We define a bounded linear operator $K_X(t,s):X\to X$ in terms of~$K(t,s)$ 
in~\eqref{eq: K def} by
\[
\iprod{K_X(t,s)v,w}=\iprod{K(t,s)v,w}
	\quad\text{for $v$, $w\in X$ and $0\le s\le t\le T$,}
\]
and let $f_X(t)$ denote the $L_2$-projection onto~$X$ of~$f(t)$ 
from~\eqref{eq: Volterra eqn prelim}, that is,
\[
\iprod{f_X(t),w}=\iprod{f(t),w}\quad\text{for $w\in X$ and $0\le t\le T$.}
\]
In this way, we arrive at a finite dimensional reduction of the 
Volterra equation~\eqref{eq: Volterra eqn},
\begin{equation}\label{eq: Volterra eqn X}
u_X(t)+\int_0^tK_X(t,s)u_X(s)\,ds=f_X(t)\quad\text{for $0\le t\le T$.}
\end{equation}
In the next theorem, we outline a self-contained proof of existence and 
uniqueness under relaxed assumptions on the coefficients in the fractional 
PDE~\eqref{eq: FPDE}. Similar results for scalar-valued kernels are shown by 
Linz~\cite[\S3.4]{Linz1985}, Becker~\cite{Becker2011}, and 
Brunner~\cite{Brunner2017}.  

Henceforth, $C$ will denote a generic constant that may depend on the 
coefficients in~\eqref{eq: FPDE}, the spatial domain~$\Omega$, the time 
interval~$[0,T]$, the fractional exponent~$\alpha$, the parameter~$\eta$, and 
the integer~$m$ in~\eqref{eq: reg coeff}.  However, any dependence on the 
subspace~$X$ is indicated explicitly by writing~$C_X$.
We let $Y=C([0,T];X)$ with the norm $\|v\|_Y=\max_{0\le t\le T}\|v(t)\|_X$.

\begin{theorem}\label{thm: uX exist unique}
Assume that the coefficients in~\eqref{eq: FPDE} satisfy
\begin{gather*}
\kappa\in L_\infty(\Omega)^{d\times d},\quad
\vec F\in W^1_\infty\bigl((0,T);L_\infty(\Omega)^d\bigr),\quad
\vec G\in L_\infty\bigl((0,T);L_\infty(\Omega)^d\bigr),\\
a\in W^1_\infty\bigl((0,T);L_\infty(\Omega)\bigr),\quad
b\in L_\infty\bigl((0,T);L_\infty(\Omega)\bigr).
\end{gather*}
Assume, in addition, that the source term $g:(0,T]\to L_2(\Omega)$ is a 
measurable function satisfying
\begin{equation}\label{eq: g bound}
\|g(t)\|\le Mt^{\eta-1}\quad\text{for $0<t\le T$,}
\end{equation}
where $M$~and $\eta$ are positive constants, and that the initial data 
$u_0\in L_2(\Omega)$. Then, the weakly-singular Volterra integral 
equation~\eqref{eq: Volterra eqn X} has a unique solution~$u_X\in Y$, and 
moreover $\|u_X\|_Y\le C_X\|f_X\|_Y\le C_X(\|u_0\|+M)$.
\end{theorem}
\begin{proof}
Our assumptions on $u_0$~and $g$ ensure that $f_X\in Y$. The 
kernel~\eqref{eq: K def} has the form
\[
K(t,s)=\omega_\alpha(t-s)G(t,s)+H(t,s),
\]
where 
\[
G(t,s)=K_1(t)-\Gamma(\alpha)(t-s)\int_0^1\omega_\alpha(y)
	K_1'\bigl(s+(t-s)y\bigr)\,dy
\]
and $H(t,s)=K_2(s)$ for $0\le s\le t\le T$.  Our assumptions on the 
coefficients of the fractional PDE~\eqref{eq: FPDE}
ensure that $G$~and $H$ are continuous mappings from the closed 
triangle~$\triangle=\{\,(t,s):0\le s\le t\le T\,\}$ into the space of bounded 
linear operators $H^1_0(\Omega)\to H^{-1}(\Omega)$. Likewise, 
\[
K_X(t,s)=\omega_\alpha(t-s)G_X(t,s)+H_X(t,s),
\]
where 
$G_X(t,s):X\to X$ and $H_X(t,s):X\to X$ are defined by
\[
\iprod{G_X(t,s)v,w}=\iprod{G(t,s)v,w}\quad\text{and}\quad
\iprod{H_X(t,s)v,w}=\iprod{H(t,s)v,w}
\]
for $(t,s)\in\triangle$~and $v$, $w\in X$. Since $X$ is finite dimensional, 
$G_X$~and $H_X$ are continuous functions from $\triangle$ into the space of 
bounded linear operators~$X\to X$.  Hence, there is a positive 
constant~$\gamma_X$ such that
\[
\|K_X(t,s)v\|_X\le\gamma_X\omega_\alpha(t-s)\|v\|_X
	\quad\text{for $(t,s)\in\triangle$~and $v\in X$,}
\]
so we can define the Volterra operator~$\mathcal{K}_X:Y\to Y$ by
\[
\mathcal{K}_Xv(t)=\int_0^tK_X(t,s)v(s)\,ds
	\quad\text{for $0\le t\le T$ and $v\in Y$.}
\]
We see that $\|\mathcal{K}_Xv\|_Y\le\gamma_X\omega_{1+\alpha}(T)\|v\|_Y$.  In
fact, using the semigroup property,
\[
\int_0^t\omega_\alpha(t-s)\omega_\beta(s)\,ds=\omega_{\alpha+\beta}(t),
\]
we obtain the following estimate for the operator norm of the $n$th power
of~$\mathcal{K}_X$,
\[
\|\mathcal{K}_X^n\|_{Y\to Y}\le\gamma_X^n\max_{0\le t\le T}
	\int_0^t\omega_{n\alpha}(t-s)\,ds=\gamma_X^n\omega_{1+n\alpha}(T)
	\quad\text{for $n\ge1$.}
\]
It follows that the 
sum~$\mathcal{R}_X=\sum_{n=1}^\infty(-1)^n\mathcal{K}_X^n$ defines a bounded 
linear operator with
\[
\|\mathcal{R}_X\|_{Y\to Y}
	\le\sum_{n=1}^\infty\omega_{1+n\alpha}(T)\gamma_X^n
	=E_\alpha(\gamma_XT^\alpha)-1.
\]
This operator is the resolvent for~$\mathcal{K}_X$, that is,
\[
u_X+\mathcal{K}_Xu_X=f_X\quad\text{if and only if}\quad
u_X=f_X-\mathcal{R}_Xf_X,
\]
implying the existence and uniqueness of~$u_X\in Y$, as well as the 
\emph{a priori} estimate claimed in the theorem.
\end{proof}

For a scalar, weakly-singular, second-kind Volterra equation,
it is known that if $f_X$ admits an expansion in powers of $t$~and $t^\alpha$, 
then so does the solution~$u_X$; see Lubich~\cite[Corollary~3]{Lubich1983}, and 
also Brunner, Pedas and Vainikko~\cite[Theorem~2.1]{BrunnerEtAl1999} (with 
$\nu=1-\alpha$).  To outline a proof that a similar result holds for 
\emph{systems} of Volterra equations, let 
$C^m_\alpha=C^m_\alpha\bigl([0,T];X\bigr)$ denote the space of continuous 
functions $v:[0,T]\to X$ that are $C^m$ on the half-open interval~$(0,T]$ and 
for which the seminorm
\[
|v|_{j,\alpha}=\sup_{0<t\le T}t^{j-\alpha}\|v^{(j)}(t)\|_X
\quad\text{is finite for~$1\le j\le m$.}
\]
We make $C^m_\alpha$ into a Banach space by defining the obvious norm:
\[
\|v\|_{m,\alpha}=\|v\|_Y+\sum_{j=1}^m|v|_{j,\alpha}.
\]

\begin{theorem}\label{thm: uX reg}
Let $m\ge1$, and strengthen the assumptions \eqref{eq: reg coeff} by requiring
\[
\vec F, \vec G\in C^{m+1}\bigl([0,T];W^1_\infty(\Omega)^d\bigr)
\quad\text{and}\quad
a, b\in C^m\bigl([0,T];L_\infty(\Omega)\bigr).
\]
If $u_0\in L_2(\Omega)$ and $g:(0,T]\to X$ is $C^m$ with
$\|g^{(i-1)}(t)\|\le Mt^{\alpha-i}$ for $1\le i\le m$, then 
$u_X\in C^m_\alpha$ and $\|u_X\|_{m,\alpha}\le C_X\|f_X\|_{m,\alpha}
\le C_X(\|u_0\|+M)$.
\end{theorem}
\begin{proof}
Our assumptions on $u_0$~and $g$ imply that $f_X\in C^m_\alpha$.  
Using the substitution~$z=s+(t-s)y$ in~\eqref{eq: K def}, we find that 
if $j+k\le m$ and $0\le s<t\le T$, then
\[
\bigl\|\partial_t^k(\partial_t+\partial_s)^jK(t,s)v\bigr\|_{H^{-1}(\Omega)}
	\le C_X(t-s)^{\alpha-1-k}\|v\|_{H^1_0(\Omega)}
	\quad\text{for $v\in H^1_0(\Omega)$,}
\]
and, since $X$ is finite dimensional,
\[
\bigl\|\partial_t^k(\partial_t+\partial_s)^jK_X(t,s)v\bigr\|_X
	\le C_X (t-s)^{\alpha-1-k}\|v\|_X
	\quad\text{for $v\in X$.}
\]
Hence, the Volterra operator~$\mathcal{K}_X:C^m_\alpha\to C^m_\alpha$ 
is compact \cite[Theorem~6.1]{Vainikko2007}. 
\cref{thm: uX exist unique} implies that the homogeneous equation, 
$u_X+\mathcal{K}_Xu_X=0$, has only the trivial
solution~$u_X=0$, and therefore the inhomogeneous 
equation~$u_X+\mathcal{K}_Xu_X=f_X$ 
is well-posed not only in~$Y$ but also in~$C^m_\alpha$.
\end{proof}

Our goal in the remainder of this section is to obtain bounds for 
$\|u_X(t)\|$ and $\|\nabla u_X(t)\|$ with constants that are independent 
of~$X$.  Our proof relies on a sequence of technical lemmas. To simplify our 
estimates, we rescale the time variable, if necessary, so that the minimal 
eigenvalue of~$\kappa$ is bounded below by unity:
\begin{equation}\label{eq: lambda min}
\lambda_{\mathrm{min}}\bigl(\kappa(x)\bigr)\ge1\quad\text{for $x\in\Omega$.}
\end{equation}
In this way, $\iprod{\kappa\nabla v,\nabla v}\ge\|\nabla v\|^2$ 
for~$v\in H^1_0(\Omega)$, and we see from~\eqref{eq: Q1 Plancherel} that for
(real-valued) $\phi\in C\bigl([0,T];H^1_0(\Omega)\bigr)$,
\begin{align*}
\int_0^t\iprod{\kappa\FI^\mu\nabla\phi,\nabla\phi}\,ds
	&=\frac{\cos(\pi\mu/2)}{\pi}\int_0^\infty y^{-\mu}
	\iprod{\kappa\nabla\hat\phi(iy),\overline{\nabla\hat\phi(iy)}}\,dy\\
	&\ge\frac{\cos(\pi\mu/2)}{\pi}\int_0^\infty y^{-\mu}
		\|\nabla\hat\phi(iy)\|^2\,dy,
\end{align*}
so
\begin{equation}\label{eq: kappa Q1}
\int_0^t\iprod{\kappa\FI^\mu\nabla\phi,\nabla\phi}\,ds
\ge\int_0^t\iprod{\FI^\mu\nabla\phi,\nabla\phi}\,ds
	=\Q^\mu_1(\nabla\phi,t).
\end{equation}

Since \eqref{eq: Volterra eqn prelim} is equivalent to~\eqref{eq: Volterra eqn},
if $v\in X$ then
\begin{align*}
\biggiprod{\int_0^t&K_X(t,s)u_X(s)\,ds,v}\\
	&=\int_0^t\bigiprod{K_1(s)\partial_s^{1-\alpha}u_X,v}\,ds
	+\int_0^t\bigiprod{K_2(s)u_X(s),v}\,ds\\
	&=\bigiprod{\kappa(\FI^\alpha\nabla u_X)(t),\nabla v}
	-\bigiprod{(B_1u_X)(t),\nabla v}+\bigiprod{(B_2u_X)(t),v},
\end{align*}
where
\begin{equation}\label{eq: vec B1 B2}
\begin{aligned}
\vec B_1\phi(t)&=\int_0^t\Bigl(\vec F(s)\partial_s^{1-\alpha}\phi(s)
    +\vec G(s)\phi(s)\Bigr)\,ds,\\
B_2\phi(t)&=\int_0^t\Bigl(a(s)\partial_s^{1-\alpha}\phi(s)
    +b(s)\phi(s)\Bigr)\,ds.
\end{aligned}
\end{equation}
Assuming $\phi\in C^1_\alpha([0,T];X\bigr)$, we may integrate by parts and use 
the notation~\eqref{eq: Yd B} to write
\begin{equation}\label{eq: B1 B2}
\vec B_1=B_{\vec F}^\alpha+B_{\vec G}^1
\quad\text{and}\quad
B_2=B_a^\alpha+B_b^1.
\end{equation}
Thus, the solution of~\eqref{eq: Volterra eqn X} satisfies
\begin{multline}\label{eq: weak uX}
\iprod{u_X(t),v}+\iprod{\kappa\nabla\FI^\alpha u_X(t),\nabla v}
    -\bigiprod{(\vec B_1u_X)(t),\nabla v}+\bigiprod{(B_2u_X)(t),v}\\
    =\iprod{f_X(t),v}\quad\text{for $v\in X$,}
\end{multline}
which yields the following estimates (with $C$ independent of~$X$).

\begin{lemma}\label{lem: uX estimate}
For $0\le t\le T$, the solution~$u_X$ of the Volterra equation~\eqref{eq: 
Volterra eqn X} satisfies the \emph{a priori} estimates
\[
\Q_1^\alpha(u_X,t)+\Q_2^\alpha(\nabla u_X,t)\le Ct^\alpha\Q^0(f_X,t)
\]
and
\[
\Q^0(u_X,t)+\Q_1^\alpha(\nabla u_X,t)\le C\Q^0(f_X,t).
\]
\end{lemma}
\begin{proof}
From~\eqref{eq: weak uX},
\begin{multline*}
\bigiprod{u_X(t),v}+\iprod{\kappa\nabla\FI^\alpha u_X(t),\nabla v}
    \le \tfrac12\|\nabla v\|^2
	+\tfrac12\|\vec B_1 u_X(t)\|^2+\tfrac12\|B_2 u_X(t)\|^2\\
	+\tfrac12\|v\|^2+\bigiprod{f_X(t),v}.
\end{multline*}
Choosing $v=\FI^\alpha u_X(t)$ we have
$\iprod{\kappa\nabla\FI^\alpha u_X(t),\nabla v}=\iprod{\kappa\nabla v,\nabla v}
\ge\|\nabla v\|^2$ because of~\eqref{eq: lambda min}.  Thus, after
canceling the term~$\tfrac12\|\nabla v\|^2$ and integrating in time, we see that
\begin{multline}\label{eq: uX nabla uX}
\Q_1^\alpha(u_X,t)+\tfrac{1}{2}\Q_2^\alpha(\nabla u_X,t)
    \le\tfrac12\Q^0(\vec B_1u_X,t)+\tfrac12\Q^0(B_2u_X,t)
	+\tfrac12\Q^\alpha_2(u_X,t)\\
    +\int_0^t\bigiprod{f_X(s),\FI^\alpha u_X(s)}\,ds.
\end{multline}
Using the representation~\eqref{eq: B1 B2} and the achieved
estimate~\eqref{eq: B1 estimate int},
\begin{align*}
\Q^0(\vec B_1u_X,t)&\le2\Q^0(B_{\vec F}^\alpha u_X,t)
	+2\Q^0(B_{\vec G}^1 u_X,t)\\
    &\le C\Q^\alpha_2(u_X,t)+C\Q^1_2(u_X,t)\le C\Q^\alpha_2(u_X,t),
\end{align*}
where, in the final step, we used \cref{lem: E}.  In the same way,
\[
\Q^0(B_2 u_X,t)\le C\Q^\alpha_2(u_X,t). 
\]
Using \eqref{eq: AC} with $\phi=f_X$, $\psi=u_X$~and $\epsilon=1/2$, we deduce 
that
\[
\Q_1^\alpha(u_X,t)+\tfrac{1}{2}\Q_2^\alpha(\nabla u_X,t)
    \le C\Q^\alpha_2(u_X,t)+Ct^\alpha\Q^0(f_X,t)+\tfrac12\Q^\alpha_1(u_X,t).
\]
Hence, applying \cref{lem: D} with~$\phi=u_X$, we can show that the
function~$\mathsf{q}(t)=\Q^\alpha_1(u_X,t)+\Q^\alpha_2(\nabla u_X,t)$ satisfies
\[
\mathsf{q}(t)\le Ct^\alpha \Q^0(f_X,t)
	+C\int_0^t\omega_\alpha(t-s)\Q^\alpha_1(u_X,s)\,ds.
\]
Since $\Q^\alpha_1(u_X,s)\le\mathsf{q}(s)$, \cref{lem: Gronwall} implies
the first estimate.

To show the second estimate, use
$-\bigiprod{(\vec B_1u_X)(t),\nabla v}=\bigiprod{\nabla\cdot\vec B_1u_X(t),v}$
in~\eqref{eq: weak uX} to obtain
\begin{multline*}
\iprod{u_X(t),v}+\bigiprod{\kappa\nabla\FI^\alpha u_X(t),\nabla v}
    \le\tfrac12\|v\|^2+\tfrac32\|\nabla\cdot(\vec B_1u_X)(t)\|^2\\
	+\tfrac32\|(B_2u_X)(t)\|^2
	+\tfrac32\|f_X(t)\|^2.
\end{multline*}
Choosing $v=u_X(t)$, integrating in time, and using \eqref{eq: kappa Q1}, we 
have
\[
\tfrac12\Q^0(u_X,t)+\Q^\alpha_1(\nabla u_X,t)
    \le C\Q^0(\nabla\cdot\vec B_1u_X,t)+C\Q^0(B_2u_X,t)+C\Q^0(f_X,t).
\]
Since
\begin{multline}\label{eq: div B}
\nabla\cdot(B^\alpha_{\vec F}u_X)(t)
    =\bigl(\nabla\cdot\vec F(t)\bigr)\FI^\alpha u_X(t)
    +\vec F(t)\cdot\FI^\alpha\nabla u_X(t)\\
    -\int_0^t\Bigl(\bigl(\nabla\cdot\vec F'(s)\bigr)\FI^\alpha u_X(s)
        +\vec F'(s)\cdot\FI^\alpha\nabla u_X(s)\Bigr)\,ds
\end{multline}
it follows that
\begin{align*}
\|\nabla\cdot(B_{\vec F}^\alpha u_X)(t)\|^2&\le C\|\FI^\alpha u_X(t)\|^2
    +C\|\FI^\alpha\nabla u_X(t)\|^2\\
	&\qquad{}+C\int_0^t\Bigl(\|\FI^\alpha u_X(s)\|^2
	+\|\FI^\alpha\nabla u_X(s)\|^2\Bigr)\,ds,
\end{align*}
implying that $\Q^0(\nabla\cdot B_{\vec F}^\alpha u_X,t)
\le C\Q_2^\alpha(u_X,t)+C\Q_2^\alpha(\nabla u_X,t)$. In the same way, 
$\Q^0(\nabla\cdot B_{\vec G}^1u_X,t)\le C\Q_2^1(u_X,t)+C\Q_2^1(\nabla u_X,t)$
and therefore, by \cref{lem: E},
\[
\Q^0(\nabla\cdot\vec B_1u_X,t)\le C\Q^\alpha_2(u_X,t)
    +C\Q^\alpha_2(\nabla u_X,t).
\]
Recall $\Q^0(B_2u_X,t)\le C\Q^\alpha_2(u_X,t)$ and let
$\mathsf{q}(t)=\Q^0(u_X,t)+\Q^\alpha_1(\nabla u_X,t)$. It follows using
\cref{lem: D} and \eqref{eq: C} that
\begin{align*}
\mathsf{q}(t)&\le C\Q^\alpha_2(u_X,t)+C\Q^\alpha_2(\nabla u_X,t)+C\Q^0(f_X,t)\\
    &\le C\Q^0(f_X,t)+C\int_0^t\omega_\alpha(t-s)
        \Bigl(\Q^\alpha_1(u_X,s)+\Q^\alpha_1(\nabla u_X,s)\Bigr)\,ds\\
    &\le C\Q^0(f_X,t)+Ct^\alpha\int_0^t\omega_\alpha(t-s)\mathsf{q}(s)\,ds.
\end{align*}
We may now apply \cref{lem: Gronwall} to complete the proof.
\end{proof}

The function~$\Mult u_X(t)=tu_X(t)$ satisfies a similar estimate to
the first one in \cref{lem: uX estimate}, but with an additional
factor~$t^2$ on the right-hand side.

\begin{lemma}\label{lem: M uX estimate}
The solution~$u_X$ of~\eqref{eq: Volterra eqn X} satisfies
\[
\Q_1^\alpha(\Mult u_X,t)+\Q_2^\alpha(\Mult\nabla u_X,t)
	\le Ct^{2+\alpha}\Q^0(f_X,t)\quad\text{for $0\le t\le T$.}
\]
\end{lemma}
\begin{proof}
Multiplying both sides of~\eqref{eq: weak uX} by~$t$, and
applying the third identity in~\eqref{eq: commutator}, we find that (since 
$\kappa$ is independent of~$t$)
\begin{multline}\label{eq: MuX}
\iprod{\Mult u_X,v}+\bigiprod{\kappa(\FI^\alpha\Mult
    +\alpha\FI^{\alpha+1})\nabla u_X,\nabla v}\\
    =\iprod{\Mult\vec B_1u_X,\nabla v}+\iprod{\Mult(f_X-B_2u_X),v},
\end{multline}
whereas integrating \eqref{eq: weak uX} in time gives
\[
\iprod{\kappa\FI^{\alpha+1}\nabla u_X,\nabla v}
	=\iprod{\FI^1\vec B_1u_X,\nabla v}
    +\bigiprod{\FI^1(f_X-u_X-B_2u_X),v},
\]
so, after eliminating $\iprod{\kappa\FI^{\alpha+1}\nabla u_X,\nabla v}$,
\begin{multline*}
\iprod{\Mult u_X,v}+\iprod{\kappa\FI^\alpha\Mult\nabla u_X,\nabla v}
    =\iprod{(\Mult-\alpha\FI^1)\vec B_1u_X,\nabla v}\\
	+\iprod{(\Mult-\alpha\FI^1)(f_X-B_2u_X)+\alpha\FI^1u_X,v}\\
    \le\tfrac12\|\nabla v\|^2
	+\tfrac12\|\vec B_3u_X\|^2+\tfrac12\|B_4u_X\|^2+\tfrac12\|v\|^2
	+\bigiprod{(\Mult-\alpha\FI^1)f_X+\alpha\FI^1 u_X,v},
\end{multline*}
where $\vec B_3\phi=(\Mult-\alpha\FI^1)\vec B_1\phi$~and
$B_4\phi=(\Mult-\alpha\FI^1)B_2$.
By choosing $v=\FI^\alpha\Mult u_X$, we have
$\iprod{\kappa\FI^\alpha\Mult\nabla u_X,\nabla v}
=\iprod{\kappa\nabla v,\nabla v}\ge\|\nabla v\|^2$ so, after
canceling the term~$\tfrac12\|\nabla v\|^2$ and integrating in time, 
\begin{align*}
\Q_1^\alpha&(\Mult u_X,t)+\tfrac12\Q_2^\alpha(\Mult\nabla u_X,t)\\
    &\le\tfrac12\Q^0(B_3u_X,t)+\tfrac12\Q^0(B_4u_X,t)
	+\tfrac12\Q_2^\alpha(\Mult u_X,t)\\
    &\qquad{}+\int_0^t\bigiprod{(\Mult-\alpha\FI^1)f_X,\FI^\alpha\Mult u_X}\,ds
	+\alpha\int_0^t\bigiprod{\FI^1 u_X,\FI^\alpha\Mult u_X}\,ds.
\end{align*}
Using \eqref{eq: AC}, we find that
\[
\int_0^t\bigiprod{(\Mult-\alpha\FI^1)f_X,\FI^\alpha\Mult u_X}\,ds
	\le Ct^\alpha\Q^0\bigl((\Mult-\alpha\FI^1)f_X,t)
		+\tfrac14\Q^\alpha_1(\Mult u_X,t)
\]
and
\[
\int_0^t\bigiprod{\FI^1 u_X,\FI^\alpha\Mult u_X}\,ds
	\le Ct^\alpha\Q^0(\FI^1 u_X,t)+\tfrac14\Q_1^\alpha(\Mult u_X,t),
\]
so
\begin{multline*}
\Q_1^\alpha(\Mult u_X,t)+\Q_2^\alpha(\Mult\nabla u_X,t)
    \le\Q^0(B_3u_X,t)+\Q^0(B_4u_X,t)\\
    +2\Q^\alpha_2(\Mult u_X,t)
	+Ct^\alpha\Q^0\bigl((\Mult-\alpha\FI^1)f_X,t)+Ct^\alpha\Q^0(\FI^1u_X,t).
\end{multline*}
Since 
\[
\vec B_3=(\Mult-\alpha\FI^1)B_{\vec F}^\alpha
+(\Mult-\alpha\FI^1)B_{\vec G}^1
\]
and
\[
B_4=(\Mult-\alpha\FI^1)B_a^\alpha+(\Mult-\alpha\FI^1)B_b^1,
\]
the estimate~\eqref{eq: B3 estimate int} gives
\begin{multline*}
\Q^0(\vec B_3u_X,t)+\Q^0(B_4u_X,t)\le Ct^2\Q^\alpha_2(u_X,t)+Ct^2\Q_2^1(u_X,t)\\
	\le Ct^2\Q^\alpha_2(u_X,t),
\end{multline*}
where, in the last step, we used \cref{lem: E} with $\mu=\alpha$~and
$\nu=1$.  We easily verify that
\[
\Q^0\bigl((\Mult-\alpha\FI^1)f_X,t)\le Ct^2\Q^0(f_X,t),
\]
and by \cref{lem: E} with $\mu=0$~and $\nu=1$,
\[
\Q^0(\FI^1u_X,t)=\Q_2^1(u_X,t)\le t^2\Q^0(u_X,t).
\]
Thus, the function
$\mathsf{q}(t)=\Q_1^\alpha(\Mult u_X,t)+\Q_2^\alpha(\Mult\nabla u_X,t)$
satisfies
\[
\mathsf{q}(t)\le Ct^2\Q^\alpha_2(u_X,t)+2\Q^\alpha_2(\Mult u_X,t)
	+Ct^{2+\alpha}\Q^0(f_X,t)+Ct^{2+\alpha}\Q^0(u_X,t).
\]
By \eqref{eq: B}~and \cref{lem: uX estimate},
\[
t^2\Q^\alpha_2(u_X,t)+t^{2+\alpha}\Q^0(u_X,t)
	\le Ct^{2+\alpha}\Q^0(u_X,t)\le Ct^{2+\alpha}\Q(f_X,t),
\]
and therefore, using \cref{lem: D} with $\phi=\Mult u_X$,
\[
\mathsf{q}(t)\le Ct^{2+\alpha}\Q^0(f_X,t)
	+C\int_0^t\omega_\alpha(t-s)\mathsf{q}(s)\,ds,
\]
The result now follows by applying \cref{lem: Gronwall}.
\end{proof}

\begin{lemma}\label{lem: (M uX)' estimate}
The solution~$u_X$ of~\eqref{eq: Volterra eqn X} satisfies, for $0\le t\le T$,
\[
\Q_1^\alpha\bigl((\Mult u_X)',t\bigr)
	+\Q_2^\alpha\bigl((\Mult\nabla u_X)',t\bigr)
	\le Ct^\alpha\Q^0(f_X,t)+Ct^\alpha\Q^0\bigl((\Mult f_X)',t\bigr).
\]
\end{lemma}
\begin{proof}
By differentiating \eqref{eq: MuX} with respect to~$t$, we have
\begin{multline}\label{eq: uX diff}
\bigiprod{(\Mult u_X)',v}+\bigiprod{\kappa\nabla(\FI^\alpha\Mult u_X)',\nabla v}
    =\bigiprod{\vec B_5u_X-\alpha\kappa\FI^\alpha\nabla u_X,\nabla v}\\
    +\bigiprod{(\Mult f_X)'-B_6u_X,v},
\end{multline}
where $\vec B_5\phi=(\Mult\vec B_1\phi)'$~and $B_6\phi=(\Mult B_2\phi)'$. Hence,
\begin{multline*}
\bigiprod{(\Mult u_X)',v}+\bigiprod{\kappa\nabla(\FI^\alpha\Mult u_X)',\nabla v}
       \le\tfrac12\|\nabla v\|^2
	+\|\vec B_5u_X\|^2+\tfrac12\|B_6u_X\|^2\\
	+\tfrac12\|v\|^2
	+C\|\FI^\alpha\nabla u_X\|^2+\bigiprod{(\Mult f_X)',v}.
\end{multline*}
Putting~$v=\FI^\alpha(\Mult u_X)'$, we can cancel $\tfrac12\|\nabla v\|^2$
because $v=(\FI^\alpha\Mult u_X)'$ by~\eqref{eq: FI omega}.  Thus, by
integrating in time and using \eqref{eq: AC} to show
\[
\int_0^t\bigiprod{(\Mult f_X)',\FI^\alpha(\Mult u_X)'}\,ds
	\le Ct^\alpha\Q^0\bigl((\Mult f_X)',t\bigr)
		+\tfrac12\Q_1^\alpha\bigl((\Mult u_X)',t\bigr),
\]
and using \eqref{eq: kappa Q1}, we arrive at the estimate
\begin{multline*}
\Q_1^\alpha\bigl((\Mult u_X)',t\bigr)
	+\Q_2^\alpha\bigl((\Mult\nabla u_X)',t\bigr)
	\le2\Q^0(\vec B_5u_X,t)+\Q^0(B_6u_X,t)\\
	+\Q^\alpha_2\bigl((\Mult u_X)',t\bigr)
	+C\Q_2^\alpha(\nabla u_X,t)+Ct^\alpha\Q^0\bigl((\Mult f_X)',t\bigr).
\end{multline*}
Since 
\[
\vec B_5u_X=(\Mult B^\alpha_{\vec F}u_X)'+(\Mult B^1_{\vec G}u_X)'
\]
and
\[
B_6u_X=(\Mult B_a^\alpha u_X)'+(\Mult B_b^1 u_X)',
\]
it follows from \eqref{eq: B5 estimate int} that
\begin{multline*}
\Q^0(\vec B_5u_X,t)+\Q^0(B_6u_X,t)\le C\Q^\alpha_2\bigl((\Mult u_X)',t\bigr)
	+C\Q^\alpha_2(\Mult u_X,t)\\
	+C\Q^\alpha_2(u_X,t).
\end{multline*}
By \cref{lem: E,lem: uX estimate,lem: M uX estimate},
\begin{align*}
\Q^\alpha_2(\Mult u_X,t)+\Q^\alpha_2(u_X,t)
	&\le Ct^\alpha\Q^\alpha_1(\Mult u_X,t)+Ct^\alpha\Q^\alpha_1(u_X,t)\\
	&\le C(t^{2+2\alpha}+t^{2\alpha})\Q^0(f_X,t)
\end{align*}
and $\Q^\alpha_2(\nabla u_X,t)\le Ct^\alpha\Q^0(f_X,t)$.  Hence, the function
\[
\mathsf{q}(t)=\Q_1^\alpha\bigl((\Mult u_X)',t\bigr)
+\Q_2^\alpha\bigl((\Mult\nabla u_X)',t\bigr)
\]
satisfies
\[
\mathsf{q}(t)\le Ct^\alpha\Q^0(f_X,t)+Ct^\alpha\Q^0\bigl((\Mult f_X)',t\bigr)
	+C\Q^\alpha_2\bigl((\Mult u_X)',t\bigr).
\]
Finally, by \cref{lem: D},
\[
\Q^\alpha_2\bigl((\Mult u_X)',t\bigr)\le C\int_0^t
	\omega_\alpha(t-s)\Q^\alpha_1\bigl((\Mult u_X)',s\bigr)\,ds
\le C\int_0^t \omega_\alpha(t-s)\mathsf{q}(s)\,ds,
\]
and the desired estimate follows by \cref{lem: Gronwall}.
\end{proof}

\begin{lemma}\label{lem: (M nabla uX)' estimate}
The solution~$u_X$ of~\eqref{eq: Volterra eqn X} satisfies, for $0\le t\le T$,
\[
\Q^0\bigl((\Mult u_X)',t\bigr)+\Q_1^\alpha\bigl((\Mult\nabla u_X)',t\bigr)
	\le C\Q^0(f_X,t)+C\Q^0\bigl((\Mult f_X)',t\bigr)
\]
\end{lemma}
\begin{proof}
Using $-\bigiprod{\vec B_5u_X,\nabla v}=\bigiprod{\nabla\cdot\vec B_5u_X(t),v}$
in~\eqref{eq: uX diff}, we obtain
\begin{multline*}
\bigiprod{(\Mult u_X)',v}
	+\bigiprod{\kappa\FI^\alpha(\Mult\nabla u_X)',\nabla v}\le\tfrac12\|v\|^2
	+2\|\nabla\cdot\vec B_5u_X\|^2+2\|B_6u_X\|^2\\
	+\|(\Mult f_X)'\|^2-\alpha\iprod{\kappa\FI^\alpha\nabla u_X,\nabla v}.
\end{multline*}
Choosing $v=(\Mult u_X)'$, integrating in time, and using 
\eqref{eq: kappa Q1} yields
\begin{multline*}
\tfrac12\Q^0\bigl((\Mult u_X)',t\bigr)
	+\Q^\alpha_1\bigl((\Mult\nabla u_X)',t\bigr)
	\le 2\Q^0\bigl(\nabla\cdot\vec B_5u_X,t\bigr)+2\Q^0(B_6u_X,t)\\
	+\Q^0\bigl((\Mult f_X)',t\bigr)
	-\alpha\int_0^t\bigiprod{(\Mult\nabla u_X)'(s),
		\kappa\FI^\alpha\nabla u_X(s)}\,ds.
\end{multline*}
Recall from~\eqref{eq: div B} that
$\nabla\cdot B^\alpha_{\vec F}\phi=B^\alpha_{\nabla\cdot\vec F}\phi
+B^\alpha_{\vec F\cdot{}}\nabla\phi$, where we have used the notation
\[
B^\alpha_{\vec F\cdot{}}\nabla\phi=\vec F(t)\cdot\FI^\alpha\nabla\phi
	-\int_0^t\vec F'(s)\cdot\FI^\alpha\nabla\phi(s)\,ds.
\]
Thus,
\begin{align*}
\nabla\cdot\vec B_5u_X&=\nabla\cdot\bigl(\Mult\vec B_1u_X\bigr)'
	=\bigl(\Mult\nabla\cdot\vec B_1u_X\bigr)'\\
	&=\bigl(\Mult\nabla\cdot B^\alpha_{\vec F}u_X\bigr)'
	+\bigl(\Mult\nabla\cdot B^1_{\vec G}u_X\bigr)'\\
	&=\bigl(\Mult B^\alpha_{\nabla\cdot\vec F}u_X\bigr)'
	+\bigl(\Mult B^\alpha_{\vec F\cdot{}}\nabla u_X\bigr)'\\
	&\qquad{}+\bigl(\Mult B^1_{\nabla\cdot\vec G}u_X\bigr)'
	+\bigl(\Mult B^1_{\vec G\cdot{}}\nabla u_X\bigr)',
\end{align*}
and so, by~\eqref{eq: B5 estimate int},
\begin{multline*}
\Q^0\bigl(\nabla\cdot\vec B_5u_X,t\bigr)+\Q^0(B_6u_X,t)
	\le C\Q_2^\alpha\bigl((\Mult u_X)',t\bigr)
	+C\Q_2^\alpha(\Mult u_X,t)\\
	+C\Q_2^\alpha(u_X,t)
	+C\Q^\alpha_2\bigl((\Mult\nabla u_X)',t\bigr)
	+C\Q^\alpha_2(\Mult\nabla u_X,t)+C\Q^\alpha_2(\nabla u_X,t).
\end{multline*}
By~\eqref{eq: A},
\[
\int_0^t\bigiprod{(\Mult\nabla u_X)'(s),\kappa\FI^\alpha\nabla u_X(s)}\,ds
	\le \tfrac12 \Q^\alpha_1\bigl((\Mult\nabla u_X)',t\bigr)
		+C\Q^\alpha_1(\nabla u_X,t),
\]
and thus the function $\mathsf{q}(t)=\Q^0\bigl((\Mult u_X)',t\bigr)
+\Q^\alpha_1\bigl((\Mult\nabla u_X)',t\bigr)$ satisfies
\begin{align*}
\mathsf{q}(t)&\le C\Q_2^\alpha\bigl((\Mult u_X)',t\bigr)
	+C\Q_2^\alpha(\Mult u_X,t)+C\Q_2^\alpha(u_X,t)\\
	&\qquad{}+C\Q^\alpha_2\bigl((\Mult\nabla u_X)',t\bigr)
	+C\Q^\alpha_2(\Mult\nabla u_X,t)+C\Q^\alpha_2(\nabla u_X,t)\\
	&\qquad{}+C\Q^0\bigl((\Mult f_X)',t\bigr)+C\Q^\alpha_1(\nabla u_X,t)\\
&\le C\Q_2^\alpha\bigl((\Mult u_X)',t\bigr)
	+Ct^\alpha\Q_1^\alpha(\Mult u_X,t)+Ct^\alpha\Q_1^\alpha(u_X,t)\\
	&\qquad{}+C\Q^\alpha_2\bigl((\Mult\nabla u_X)',t\bigr)
	+Ct^{2+\alpha}\Q^0(f_X,t)+Ct^\alpha\Q^0(f_X,t)\\
	&\qquad{}+C\Q^0\bigl((\Mult f_X)',t\bigr)+C\Q^0(f_X,t),
\end{align*}
where, in the second step, we used 
\cref{lem: D,lem: uX estimate,lem: M uX estimate}.  A further 
application of \cref{lem: uX estimate,lem: M uX estimate} yields
\begin{multline*}
\mathsf{q}(t)\le C\Q^0\bigl((\Mult f_X)',t\bigr)+C\Q^0(f_X,t)
	+C\Q_2^\alpha\bigl((\Mult u_X)',t\bigr)\\
	+C\Q^\alpha_2\bigl((\Mult\nabla u_X)',t\bigr).
\end{multline*}
\cref{lem: D} implies that $\Q_2^\alpha\bigl((\Mult u_X)',t\bigr)
+\Q^\alpha_2\bigl((\Mult\nabla u_X)',t\bigr)$ is bounded by
\begin{multline*}
C\int_0^t\omega_\alpha(t-s)\Bigl(
	\Q_1^\alpha\bigl((\Mult u_X)',s\bigr)
	+\Q^\alpha_1\bigl((\Mult\nabla u_X)',s\bigr)\Bigr)\,ds\\
	\le C\int_0^t\omega_\alpha(t-s)\mathsf{q}(s)\,ds,
\end{multline*}
where we used $\Q_1^\alpha\bigl((\Mult u_X)',s\bigr)
\le Ct^\alpha\Q^0 \bigl((\Mult u_X)',s\bigr)$, which follows by \cref{lem: E}.  
Finally, \cref{lem: Gronwall} implies the desired estimate.
\end{proof}

The preceding lemmas yield the main result for this section.

\begin{theorem}\label{thm: uX grad uX bound}
Assume that the coefficients satisfy~\eqref{eq: reg coeff}, that
the initial data $u_0\in L_2(\Omega)$ and that the source term
satisfies~\eqref{eq: g bound}.  Then, the solution~$u_X$ of the projected
Volterra equation~\eqref{eq: Volterra eqn X} satisfies (with $C$ independent
of~$X$)
\[
\|u_X(t)\|^2+t^\alpha\|\nabla u_X(t)\|^2
	\le C\bigl(\|u_0\|^2+M^2t^{2\eta}\bigr)\quad\text{for $0\le t\le T$.}
\]
\end{theorem}
\begin{proof}
Applying \cref{lem: pointwise bound} with~$\phi=\Mult u_X$, we see that
\cref{lem: (M uX)' estimate} gives
\begin{align*}
t^2\|u_X(t)\|^2&=\|\Mult u_X(t)\|^2
	\le Ct^{1-\alpha}\Q^\alpha_1\bigl((\Mult u_X)',t\bigr)\\
	&\le Ct\Q^0(f_X,t)+Ct\Q^0\bigl((\Mult f_X)',t\bigr).
\end{align*}
Define $g_X:[0,T]\to X$ by~$\iprod{g_X(t),v}=\iprod{g(t),v}$ for $v\in X$, and 
observe that $f_X=u_0+\FI^1g_X$ and $(\Mult f_X)'=f_X+\Mult f_X'=f_X+\Mult g_X$. 
We find using \eqref{eq: g bound} that
\begin{equation}\label{eq: hX u0 g}
\begin{aligned}
\Q^0(f_X,t)+\Q^0\bigl((\Mult f_X)',t\bigr)
&\le C\int_0^t\biggl(\|u_0\|^2+\|\FI^1g\|^2+\|\Mult g\|^2\biggr)\,ds\\
&\le Ct\bigl(\|u_0\|^2+M^2t^{2\eta}\bigr),
\end{aligned}
\end{equation}
so the estimate for the first term~$\|u_X(t)\|^2$ follows at once. Similarly,
applying \cref{lem: pointwise bound} with~$\phi=(\Mult\nabla u_X)'$
followed by \cref{lem: (M nabla uX)' estimate}, we have
\begin{align*}
t^{2+\alpha}\|\nabla u_X(t)\|&=t^\alpha\|\Mult\nabla u_X(t)\|^2
	\le Ct\Q^\alpha_1\bigl((\Mult\nabla u_X)',t\bigr)\\
	&\le Ct\Q^0(f_X,t)+Ct\Q^0\bigl((\Mult f_X)',t\bigr),
\end{align*}
implying the estimate for the second term~$t^\alpha\|\nabla u_X(t)\|^2$.
\end{proof}
\section{The weak solution}\label{sec: weak}
We will now establish that the weak formulation~\eqref{eq: weak} of the 
initial-boundary value problem
\eqref{eq: FPDE}--\eqref{eq: initial condition} is well-posed.
The proof relies on our estimates from \cref{sec: projected} and also
the following local H\"older continuity properties of~$u_X$.

\begin{lemma}\label{lem: Holder}
If $0<\delta\le t_1<t_2\le T$, then
\[
\|u_X(t_2)-u_X(t_1)\|^2
	\le C\delta^{-2}t_2\bigl(\|u_0\|^2 +M^2t_2^{2\eta}\bigr)(t_2-t_1)
\]
and
\[
\|\FI^\alpha\nabla u_X(t_2)-\FI^\alpha\nabla u_X(t_1)\|
	\le C\bigl(\|u_0\|+Mt_2^\eta\bigr)\bigl[
		\delta^{\alpha-2}(t_2-t_1)+\delta^{-\alpha/2}(t_2-t_1)^\alpha\bigr].
\]
\end{lemma}
\begin{proof}
The Cauchy--Schwarz inequality implies that
\[
\|u_X(t_2)-u_X(t_1)\|^2=\biggl\|\int_{t_1}^{t_2}u_X'(s)\,ds\biggr\|^2
	\le(t_2-t_1)\int_{t_1}^{t_2}\|u_X'(s)\|^2\,ds,
\]
and by the second inequality of~\cref{lem: uX estimate}, together 
with \cref{lem: (M nabla uX)' estimate},
\begin{align*}
\int_{t_1}^{t_2}\|u_X'(s)\|^2\,ds
	&=\int_{t_1}^{t_2}s^{-2}\|(\Mult u_X)'(s)-u_X(s)\|^2\,ds\\
	&\le2\delta^{-2}\int_0^{t_2}\bigl(\|(\Mult u_X)'\|^2+\|u_X\|^2\bigr)\,ds\\
	&=2\delta^{-2}\bigl[\Q^0\bigl((\Mult u_X)',t_2\bigr)+\Q^0(u_X,t_2)\bigr]\\
	&\le C\delta^{-2}\bigl[\Q^0\bigl(\Mult f_X)',t_2\bigr)
		+\Q^0(f_X,t_2)\bigr].
\end{align*}
The first result now follows from~\eqref{eq: hX u0 g}. To prove the second, we write
\begin{multline*}
\FI^\alpha\nabla u_X(t_2)-\FI^\alpha\nabla u_X(t_1)
	=\int_0^{t_1-\delta/2}\bigl[\omega_\alpha(t_2-s)-\omega_\alpha(t_1-s)\bigr]
	\nabla u_X(s)\,ds\\
	+\int_{t_1-\delta/2}^{t_1}\bigl[
	\omega_\alpha(t_2-s)-\omega_\alpha(t_1-s)\bigr]\nabla u_X(s)\,ds
	+\int_{t_1}^{t_2}\omega_\alpha(t_2-s)\nabla u_X(s)\,ds,
\end{multline*}
and deduce from \cref{thm: uX grad uX bound} that
\[
\|\FI^\alpha\nabla u_X(t_2)-\FI^\alpha\nabla u_X(t_1)\|
	\le C\bigl(\|u_0\|+Mt_2^\eta\bigr)\bigl(I_1+I_2+I_3),
\]
where
\begin{align*}
I_1&=\int_0^{t_1-\delta/2}\bigl[\omega_\alpha(t_1-s)-\omega_\alpha(t_2-s)\bigr]
	s^{-\alpha/2}\,ds,\\
I_2&=\int_{t_1-\delta/2}^{t_1}
\bigl[\omega_\alpha(t_1-s)-\omega_\alpha(t_2-s)\bigr]s^{-\alpha/2}\,ds,\\
I_3&=\int_{t_1}^{t_2}\omega_\alpha(t_2-s)s^{-\alpha/2}\,ds.
\end{align*}
By the mean value theorem,
\[
\omega_\alpha(t_1-s)-\omega_\alpha(t_2-s)
	=(t_2-t_1)|\omega_{\alpha-1}(\xi)|\quad\text{with $t_1-s<\xi<t_2-s$,}
\]
and if $0<s<t_1-\delta/2$ then $t_1-s>\delta/2$ so
\begin{align*}
I_1&\le(t_2-t_1)|\omega_{\alpha-1}(\delta/2)|
	\int_0^{t_1-\delta/2}\frac{ds}{s^{\alpha/2}}\\
	&\le\biggl(\frac{2}{\delta}\biggr)^{2-\alpha}\,\frac{1-\alpha}{1-\alpha/2}\,
	\frac{(t_1-\delta/2)^{1-\alpha/2}}{\Gamma(\alpha)}\,(t_2-t_1).
\end{align*}
Moreover,
\begin{align*}
I_2&\le(\delta/2)^{-\alpha/2}\int_{t_1-\delta/2}^{t_1}
	\bigl[\omega_\alpha(t_1-s)-\omega_\alpha(t_2-s)\bigr]\,ds\\
	&=(2/\delta)^{\alpha/2}\bigl[
	\omega_{\alpha+1}(t_2-t_1)+\omega_{\alpha+1}(\delta/2)
		-\omega_{\alpha+1}(t_2-t_1+\delta/2)\bigr]\\
	&\le(2/\delta)^{\alpha/2}\omega_{\alpha+1}(t_2-t_1)
\end{align*}
and $I_3\le\delta^{-\alpha/2}\int_{t_1}^{t_2}\omega_\alpha(t_2-s)\,ds
=\delta^{-\alpha/2}\omega_{\alpha+1}(t_2-t_1)$.
\end{proof}

Our existence theorem is stated as follows.  Note the weak continuity
at~$t=0$ asserted in part~5; we show in the companion 
paper~\cite{McLeanEtAl2019} that the solution~$u$ is continuous on 
the closed interval~$[0,T]$ provided $u_0\in\dot H^\mu(\Omega)$ for 
some~$\mu>0$.

\begin{theorem}\label{thm: existence}
Assume that the coefficients satisfy~\eqref{eq: reg coeff}, 
that the source term satisfies \eqref{eq: g bound}, and that the initial data
$u_0\in L_2(\Omega)$. Then, the initial-boundary value
problem~\eqref{eq: FPDE}--\eqref{eq: initial condition} has a weak solution~$u$.
More precisely, there exists a function $u:[0,T]\to L_2(\Omega)$ with the
following properties.
\begin{enumerate}
\item The restriction~$u:(0,T]\to L_2(\Omega)$ is continuous.
\item If $0<t\le T$, then
$u(t)\in H^1_0(\Omega)$ with 
\[
\|u(t)\|+t^{\alpha/2}\|\nabla u(t)\|
\le C\bigl(\|u_0\|+Mt^\eta\bigr).
\]
\item The functions $\FI^\alpha u$~and $B_2u$ are continuous from the closed
interval~$[0,T]$ to $L_2(\Omega)$. Likewise, $\FI^\alpha\nabla u$~and $\vec
B_1u$ are continuous from~$[0,T]$ to~$L_2(\Omega)^d$.
\item At $t=0$ we have $\FI^\alpha u=B_2u=0$, $\FI^\alpha\nabla u=\vec B_1u=0$
and $u(0)=u_0$.
\item If $t\to0$, then $\iprod{u(t),v}\to\iprod{u(0),v}$ for each 
$v\in L_2(\Omega)$. 
\item For $0\le t\le T$~and $v\in H^1_0(\Omega)$,
\begin{equation}\label{eq: weak u}
\iprod{u(t),v}+\bigiprod{\kappa(\FI^\alpha\nabla u)(t),\nabla v}
    -\bigiprod{(\vec B_1u)(t),\nabla v}+\iprod{(B_2u)(t),v}=\iprod{f(t),v}.
\end{equation}
\end{enumerate}
\end{theorem}
\begin{proof}
Let $\psi_1$, $\psi_2$, $\psi_3$, \dots be a sequence of functions spanning a
dense subspace of~$H^1_0(\Omega)$. For each integer~$n\ge1$, let
$X_n=\vecspan\{\psi_1,\psi_2,\ldots,\psi_n\}$ and for brevity denote the
solution of~\eqref{eq: weak uX} with~$X=X_n$ by~$u_n=u_X$, and likewise write
$f_n=f_X$, so that
\begin{equation}\label{eq: weak un}
\iprod{u_n(t),v}+\iprod{\kappa(\FI^\alpha\nabla u_n)(t),\nabla v}
    -\iprod{(B_1u_n)(t),\nabla v}+\iprod{(B_2u_n)(t),v}=\iprod{f_n(t),v}
\end{equation}
for $v\in X_n$ and $0<t\le T$. We see from
\cref{thm: uX grad uX bound}~and \cref{lem: Holder} that, 
whenever~$0<\delta<T$, the sequence of functions~$u_n$ is bounded and 
equicontinuous in~$C\bigl([\delta,T];L_2(\Omega)\bigr)$. By
choosing a sequence of values of~$\delta$ tending to zero we can select a
subsequence, again denoted by~$u_n$, such that $u_n(t)$ converges
in~$L_2(\Omega)$ for~$0<t\le T$.  We may therefore define
\[
u(t)=\lim_{n\to\infty}u_n(t)\quad\text{for $0<t\le T$,}
\]
and this function satisfies property~1 because, given any 
fixed~$\delta\in(0,T)$, the limit is uniform for~$t\in[\delta,T]$.
Similarly, the functions~$\FI^\alpha\nabla u_n$ are bounded and equicontinuous
in~$C\bigl([\delta,T];L_2(\Omega)^d\bigr)$ so
$\FI^\alpha\nabla u:(0,T]\to L_2(\Omega)^d$ is continuous.  In fact, it will
follow from~\eqref{eq: FI grad u(t) estimate} below that
$\|\FI^\alpha\nabla u(t)\|\to0$ as~$t\to0$, so 
$\FI^\alpha\nabla u:[0,T]\to L_2(\Omega)^d$ is continuous.

By \cref{thm: uX grad uX bound},
\[
\|u_n(t)\|\le C\bigl(\|u_0\|+Mt^\eta\bigr)
	\quad\text{for $0<t\le T$,}
\]
so by sending $n\to\infty$ we conclude that
$\|u(t)\|\le C\bigl(\|u_0\|+Mt^\eta\bigr)$.  Also, for $0<t\le T$,
\[
|\iprod{u_n(t),v}|\le C\|u_n(t)\|_{H^1_0(\Omega)}\|v\|_{H^{-1}(\Omega)}
	\le Ct^{-\alpha/2}\bigl(\|u_0\|+Mt^\eta\bigr)\|v\|_{H^{-1}(\Omega)}
\]
and sending $n\to\infty$ it follows that
\[
|\iprod{u(t),v}|
  \le Ct^{-\alpha/2}\bigl(\|u_0\|+Mt^\eta\bigr)\|v\|_{H^{-1}(\Omega)}
  \quad\text{for all $v\in L_2(\Omega)$,}
\]
so $u(t)\in H^1_0(\Omega)$ with
$\|u(t)\|_{H^1_0(\Omega)}\le Ct^{-\alpha/2}(\|u_0\|+Mt^\eta\bigr)$,
establishing property~2.

Since $\|u(t)\|$ is bounded, $\FI^\alpha u$ is continuous on~$[0,T]$ with
\begin{equation}\label{eq: FI u(t) estimate}
\begin{aligned}
\|\FI^\alpha u(t)\|&\le\int_0^t\omega_\alpha(t-s)\|u(s)\|\,ds\\
	&\le C\int_0^t(t-s)^{\alpha-1}\bigl(\|u_0\|+Ms^\eta\bigr)\,ds
	\le C\bigl(\|u_0\|+Mt^\eta\bigr)t^\alpha,
\end{aligned}
\end{equation}
and similarly
\begin{equation}\label{eq: FI grad u(t) estimate}
\|\FI^\alpha\nabla u(t)\|
	\le C\int_0^t(t-s)^{\alpha-1}s^{-\alpha/2}
        \bigl(\|u_0\|+Ms^\eta\bigr)\,ds
	\le C(\|u_0\|+Mt^\eta\bigr)t^{\alpha/2}.
\end{equation}
Likewise, for $n\ge1$,
\begin{equation}\label{eq: FI un(t) grad un(t) estimate}
\|\FI^\alpha u_n(t)\|\le C\bigl(\|u_0\|+Mt^\eta\bigr)t^\alpha
\quad\text{and}\quad
\|\FI^\alpha\nabla u_n(t)\|\le C\bigl(\|u_0\|+Mt^\eta\bigr)t^{\alpha/2}.
\end{equation}
Continuity of $\vec B_1u$~and $B_2u$ follow from \eqref{eq: Yd B}~and
\eqref{eq: B1 B2}, completing the proof of property~3, with
\begin{equation}\label{eq: B1 B2 u}
\begin{aligned}
\|(\vec B_1u)(t)\|+\|(B_2u)(t)\|&\le C\|(\FI^\alpha u)(t)\|
	+C\int_0^t\bigl(\|(\FI^\alpha u)(s)\|+\|u(s)\|\bigr)\,ds\\
	&\le C\bigl(\|u_0\|+M\bigr)t^\alpha.
\end{aligned}
\end{equation}
Property~4 follows from the estimates \eqref{eq: FI u(t) estimate},
\eqref{eq: FI grad u(t) estimate}~and \eqref{eq: B1 B2 u}.

If $0\le\delta<t\le T$, then
\begin{align*}
\|&(\FI^\alpha u_n)(t)-(\FI^\alpha u)(t)\|\le\int_0^t\omega_\alpha(t-s)
    \|u_n(s)-u(s)\|\,ds\\
    &\le C\int_0^\delta(t-s)^{\alpha-1}\bigl(\|u_0\|+Ms^\eta\bigr)\,ds
    +\int_\delta^t(t-s)^{\alpha-1}\|u_n(s)-u(s)\|\,ds\\
    &\le C\delta^\alpha\bigl(\|u_0\|+M\delta^\eta\bigr)
        +\alpha^{-1}(t-\delta)^\alpha\max_{\delta\le s\le t}\|u_n(s)-u(s)\|,
\end{align*}
showing that $\FI^\alpha u_n(t)\to\FI^\alpha u(t)$ in~$L_2(\Omega)$, uniformly
for~$t\in[\delta,T]$. In fact, the convergence is uniform for~$t\in[0,T]$,
owing to the estimates \eqref{eq: FI u(t) estimate}~and
\eqref{eq: FI un(t) grad un(t) estimate}. Therefore, we see using
\eqref{eq: Yd B}~and \eqref{eq: B1 B2} that, for $v\in H^1_0(\Omega)$,
\[
\bigiprod{(\vec B_1u_n)(t),\nabla v}\to\bigiprod{(\vec B_1u)(t),\nabla v}
\quad\text{and}\quad
\bigiprod{(B_2u_n)(t),v}\to\bigiprod{(B_2u)(t),v}.
\]
Since $\iprod{f_n,\psi_j}=\iprod{f,\psi_j}$
for~$j\le n$, we have
\[
\lim_{n\to\infty}\iprod{f_n(t),\psi_j}=\iprod{f(t),\psi_j}
	\quad\text{for all $j\ge1$ and $0\le t\le T$,}
\]
and therefore $\iprod{f_n(t),v}\to\iprod{f(t),v}$ for all $v\in L_2(\Omega)$.
Thus, by sending $n\to\infty$ in~\eqref{eq: weak un}, it follows that
\eqref{eq: weak u} holds for $v\in H^1_0(\Omega)$ and $0<t\le T$.  In light of
\eqref{eq: B1 B2 u}~and \eqref{eq: FI grad u(t) estimate}, the variational
equation~\eqref{eq: weak u} is satisfied when~$t=0$ if and only if
$\iprod{u(0),v}=\iprod{u_0,v}$ for all $v\in H^1_0(\Omega)$, which is the case
if and only if we define $u(0)=u_0$.  Moreover, if $t\to0$ then
$\iprod{u(t),v}\to\iprod{f(0),v}=\iprod{u_0,v}$, for each $v\in H_0^1(\Omega)$, 
and hence by density for each $v\in L_2(\Omega)$, establishing properties 5~and 
6.
\end{proof}

\begin{remark}
Since our estimates rely on \cref{lem: alpha dep}, the constant~$C$ in part~2 
of \cref{thm: existence} becomes unbounded as~$\alpha\to1$.  However, this 
behavior appears to be an artifact of our method of proof. In the 
limiting case when~$\alpha=1$ and \eqref{eq: FPDE} reduces to a parabolic PDE, 
a simple energy argument combined with the classical Gronwall inequality
yields the \emph{a priori} estimate
\[
\|u(t)\|\le C\biggl(\|u_0\|+\int_0^t\|g(s)\|\,ds\biggr)
	\quad\text{for $0\le t\le T$;}
\]
see also the alternative analysis~\cite{LeMcLeanStynes2018} of the fractional 
Fokker--Planck equation.
\end{remark}

\begin{theorem}\label{thm: uniqueness}
The weak solution of the initial-boundary value
problem~\eqref{eq: FPDE}--\eqref{eq: initial condition} is unique.  More
precisely, under the same assumptions as \cref{thm: existence}, there is
at most one function~$u$ that satisfies \eqref{eq: weak u}
and is such that $u$~and $\FI^\alpha u$ belong
to~$L_2\bigl((0,T);L_2(\Omega)\bigr)$, and $\FI^\alpha\nabla u$ belongs
to~$L_2\bigl((0,T);L_2(\Omega)^d\bigr)$.
\end{theorem}
\begin{proof}
The problem is linear, so it suffices to show that if $u_0=0$~and
$g(t)\equiv0$ then $u(t)\equiv0$.  Thus, suppose that
\[
\iprod{u(t),v}+\bigiprod{\kappa(\FI^\alpha\nabla u)(t),\nabla v}
    -\bigiprod{(\vec B_1u)(t),\nabla v}+\iprod{(B_2u)(t),v}=0
\]
for $0<t\le T$ and $v\in H^1_0(\Omega)$.  Proceeding as in the proof 
of~\eqref{eq: uX nabla uX}, we have
\begin{multline*}
\Q^\alpha_1(u,t)+\tfrac12\Q^\alpha_2(\nabla u,t)
	\le\tfrac12\Q^0(\vec B_1u,t)+\tfrac12\Q^0(B_2u,t)+\tfrac12\Q^\alpha_2(u,t)\\
	\le C\Q^\alpha_2(u,t),
\end{multline*}
where the final step used \eqref{eq: Yd B}, \eqref{eq: B1 estimate int}~and
\cref{lem: E}.  Thus, applying \cref{lem: D}, the function
$\mathsf{q}(t)=\Q^\alpha_1(u,t)+\Q^\alpha_2(\nabla u,t)$ satisfies
\[
\mathsf{q}(t)\le C\Q^\alpha_2(\nabla u,t)
  \le C\int_0^t\omega_\alpha(t-s)\mathsf{q}(s)\,ds,
\]
and therefore $\mathsf{q}(t)=0$ for~$0\le t\le T$ by \cref{lem: Gronwall}.  In
particular, $\Q^\alpha_1(u,T)=0$, so if we put $u(t)=0$ for~$t>T$ then the 
Laplace transform of~$u$ satisfies $\hat u(iy)=0$ for~$-\infty<y<\infty$ 
by~\eqref{eq: Q1 Plancherel},
implying that $u(t)=0$ for~$0\le t\le T$.
\end{proof}
\section*{Acknowledgements}
The authors thank the University of New South Wales (Faculty 
Research Grant ``Efficient numerical simulation of anomalous transport phenomena''),
the King Fahd University of Petroleum and Minerals (project No.~KAUST005) and
the King Abdullah University of Science and Technology.
\renewcommand{\bibliofont}{\normalfont\normalsize}
\bibliographystyle{plain}
\bibliography{McLeanMustaphaAliKnio_refs}
\bigskip\smallskip
{\it
\noindent $^1$ School of Mathematics and Statistics\\
The University of New South Wales\\ Sydney 2052, Australia\\[4pt]
e-mail: w.mclean@unsw.edu.eu\\[12pt]
\noindent $^2$ Department of Mathematics and Statistics\\
 KFUPM, Dhahran 31261, Saudi Arabia\\[4pt]
e-mail: kassem@kfupm.edu.su, g201305090@kfupm.edu.sa \\[12pt]
\noindent $^3$ Computer, Electrical, Mathematical Sciences and Engineering Division\\
KAUST, Thuwal 23955, Saudi Arabia\\[4pt]
e-mail: Omar.Knio@kaust.edu.sa}
\end{document}